\documentclass[11pt]{article}
\usepackage{mylatexsetting}
\def\mytitle{Stochastic Ordering under Weaker Likelihood-Ratio Shape Conditions}

\def\myabstract{

We show that the shape hypothesis on a likelihood ratio can be weakened while retaining endpoint criteria for the hazard-rate and usual stochastic orders. The endpoint reduction persists under unimodality of the likelihood ratio and under a sign-pattern condition on the likelihood ratio minus one, with at most two sign changes and a negative right tail. It also follows from a direct superlevel-set criterion involving the same expression, which is useful in particular for discontinuous likelihood ratios.

}

\title{\mytitle}
\author{Zakaria Derbazi\footnote{Correspondence address: z.derbazi@qmul.ac.uk}\\{\it\small Queen Mary University of London}}
\date{\today}

\begin{document}

\maketitle

\begin{abstract}
\myabstract
\end{abstract}




The theory of stochastic ordering provides a standard framework for comparing probability laws, with applications across reliability, queueing theory, actuarial mathematics, and applied probability~\cite{ShakedShanthikumar,MullerStoyan}. For direct comparison between two laws, the natural object is the likelihood ratio. Its monotonicity gives the likelihood-ratio order, and concavity of its logarithm gives relative log-concavity~\cite{Whitt1985,Yu2009}. The question addressed here is how far that shape hypothesis can be weakened while retaining usable endpoint criteria for the hazard-rate and usual stochastic orders.

Writing $\ell$ for the likelihood ratio, the pairwise shape hypotheses underlying these results are summarised in Figure~\ref{fig:shape-chain}.

\begin{figure}[H]
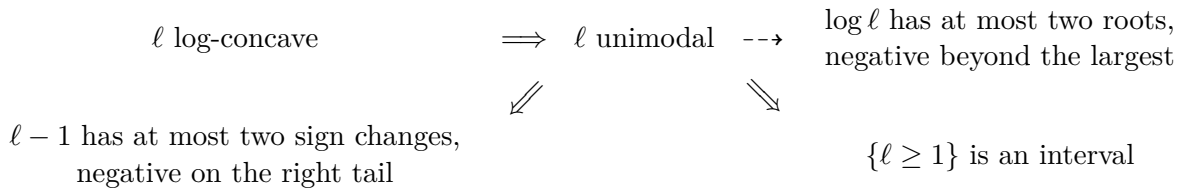

\centering
$\begin{array}{ccccc}
\text{$\ell$ log-concave}
& \Longrightarrow &
\text{$\ell$ unimodal}
& \dashrightarrow &
\begin{array}{c}
\text{$\log\ell$ has at most two roots,}\\
\text{negative beyond the largest}
\end{array} \\[1.5ex]
& \rotatebox[origin=c]{-135}{$\Longrightarrow$} & & \rotatebox[origin=c]{-45}{$\Longrightarrow$} & \\[0.5ex]
\begin{array}{c}
\text{$\ell-1$ has at most two sign changes,}\\
\text{negative on the right tail}
\end{array}
& & & &
\{\ell\ge 1\}\ \text{is an interval}
\end{array}$
\caption{Pairwise shape hypotheses on $\ell$. The dashed arrow marks a non-implication.}
\label{fig:shape-chain}
\end{figure} The first two hypotheses are classical~\cite{Whitt1985,Yu2009,BelzunceMartinezPereda2022}, while the third is the endpoint-reduction hypothesis of~\cite{YaoLouPan2023} for nonnegative laws. The present note furnishes two main results in this setting. We give a sign-pattern criterion on $\ell-1$ that delivers the same endpoint reduction on supports with finite left endpoint. Moreover, it works with the sign pattern rather than a literal root count, and therefore avoids difficulties caused by multiple zeros of $\log\ell$. Finally, we give a direct superlevel-set criterion on \(\{\ell\ge 1\}\), which dispenses with the continuity hypothesis on \(\ell\) and is convenient for likelihood ratios with a jump across the level \(1\).

Building on these observations, the paper's main contributions are:
\begin{enumerate}[label=\textup{(\roman*)}]

\item Endpoint reductions for the hazard-rate and usual stochastic orders under relative log-concavity and under unimodality of the likelihood ratio, together with the corresponding discrete and continuous endpoint tests for the likelihood-ratio order under relative log-concavity (Theorem~\ref{thm:unimodal-criterion}, Corollaries~\ref{cor:lc-endpoints} and~\ref{cor:lc-lr}).

\item An endpoint reduction on support intervals, phrased in terms of the sign pattern of \(\ell-1\), extending the result of~\cite{YaoLouPan2023} from nonnegative laws to ordered supports with finite left endpoint, together with a direct superlevel-set alternative that drops the continuity hypothesis on \(\ell\) and covers discontinuous likelihood-ratio comparisons (Lemma~\ref{lem:tail-mean-sign}, Propositions~\ref{prop:ylp-kernel} and~\ref{prop:ylp-one-sided-score}).

\end{enumerate}



\section{Preliminaries}
\label{sec:prelim}

Let \(E\) be a totally ordered support, either an integer interval or an interval of \(\R\), equipped with counting measure or Lebesgue measure according to the case. All probability measures considered below are dominated by the corresponding reference measure \(\mu\). For a probability measure  \(P\) on \(E\), write \(f_P:=\dx[]P/\dx[\mu]\), and similarly for \(Q\).

For a pair of probability measures \(P\) and \(Q\), set \(\rho:=P+Q\) and \(S:=\supp(\rho)\). The likelihood ratio is
\begin{equation}
    \label{def.ell}
    \ell(x):=\frac{\dx[]P/\dx[]\rho}{\dx[]Q/\dx[]\rho}, \qquad x\in S,
\end{equation}
with the conventions $a/0:=+\infty$ for $a>0$ and $0/0:=0$. The likelihood ratio coincides with the $\mu$-density ratio $f_P/f_Q$ wherever both $\mu$-densities are positive.

We use \(X,Y\) for random variables with values in \(E\), distributed according to the law under consideration, and \(\E\) for expectation under the law currently in scope. When the probability measure needs to be made explicit, we write \(\E_P,\E_Q\).

When \(E\) is an integer interval, we write \(\N_0:=\{0\}\cup\N\) and refer to sets of the form \([a,b]:=\{a,\dots,b\}\) or \([a,\infty):=\{a,a{+}1,\dots\}\) as intervals in \(\N_0\).  Forward differences are \(\Delta h(k):=h(k+1)-h(k)\) and \(\Delta^2 h(k):=h(k+2)-2h(k+1)+h(k)\); we call \(k\) admissible for \(\Delta^2\) when \(k,k+1,k+2\) all lie in the relevant support. We write \(\psi=\Gamma'/\Gamma\) for the digamma function.  The indicator of a set \(A\subseteq E\) is \(\one\{A\}\).

\begin{defn}[Unimodality] \label{def:unimodal}
A real-valued function \(a\) is \emph{unimodal} on an interval \(J\subseteq E\) if there exists \(x_0\in J\) such that \(a\) is nondecreasing on \(J\cap(-\infty,x_0]\) and nonincreasing on \(J\cap[x_0,\infty)\).  On an integer interval this is the usual notion of a unimodal sequence.
\end{defn}

\begin{defn}[Log-concavity]
A positive function \(a:J\to(0,\infty)\) on an interval \(J\subseteq E\) is \emph{log-concave} if \(\log a\) is concave on \(J\). On an integer interval this is equivalent to
\[
a_k^2\ge a_{k-1}a_{k+1},
\]
for every admissible \(k\in J\), or equivalently to \(\Delta^2\log a(k)\le 0\).
\end{defn}

\begin{defn}\label{def:orders}
Let \(P\) and \(Q\) be probability measures on \(E\), with survival functions \(\bar F_P(x):=P([x,\infty)\cap E)\) and $\bar F_Q$ analogously.
\begin{enumerate}[label=\textup{(\roman*)}]

\item \emph{Usual stochastic order}: \(P\st Q\) if
\(\bar F_P(x)\le\bar F_Q(x)\) for every \(x\in E\); equivalently, \(\int\varphi\,\dx[]P\le\int\varphi\,\dx[]Q\) for every bounded nondecreasing \(\varphi:E\to\R\).

\item \emph{Hazard-rate order}: writing \(h_P(x):=f_P(x)/\bar F_P(x)\) for the hazard-rate when \(\bar F_P(x)>0\), \(P\hr Q\) if the survival ratio \(\bar F_P(x)/\bar F_Q(x)\) is nonincreasing in \(x\), equivalently \(h_P(x) \ge h_Q(x)\), wherever both expressions are defined.

\item \emph{Likelihood-ratio order}: \(P\lr Q\) if the likelihood ratio \(\ell\) is nonincreasing on \(\supp(P)\cup\supp(Q)\) \cite{KlenkeMattner2010, ShakedShanthikumar}.

\item \emph{Relative log-concavity}~\cite{Whitt1985,Yu2009}: when \(\supp(P)\subseteq\supp(Q)\) and \(\supp(P)\) is an interval in \(E\), \(P\lc Q\) if \(\log \ell\) is concave on \(\supp(P)\).  In the discrete case this is equivalent to \(\Delta^2\log\ell(k)\le 0\) for every admissible  \(k\in\supp(P)\), while in the continuous case it is equivalent to \((\log\ell)''(x)\le 0\) wherever the second derivative exists.
\end{enumerate}
\end{defn}


\section{Endpoint Criteria under Weaker Shape Conditions}
\label{sec:kernel}

Throughout this section, $P$ and $Q$ are two laws whose supports form, or are contained in, an interval $J\subseteq E$, with densities $f_P, f_Q$ positive on the relevant support, and $\ell$ denotes the likelihood ratio of~\eqref{def.ell}. Each result restates the precise support hypothesis it needs (common support, support union, or support inclusion). The first step is immediate from the definitions:
\[
P\lr Q\iff \ell\text{ is nonincreasing on }J,
\qquad
P\lc Q\iff \log\ell\text{ is concave on }J.
\]
In the discrete case these become \(\Delta\log\ell(k)\le 0\) and \(\Delta^2\log\ell(k)\le 0\), while in the continuous case they become \(\partial_x\log\ell(x)\le 0\) and \(\partial_x^2\log\ell(x)\le 0\) wherever the derivatives exist.

\subsection{Unimodality and relative log-concavity}
\label{sec:comparison-kernel}
When the likelihood ratio is only known to be unimodal, the hazard-rate and usual stochastic orders still reduce to a left-endpoint test.

\begin{thm}
\label{thm:unimodal-criterion}
Let \(P,Q\) be probability measures on \(E\), and let \(\ell\) be the likelihood ratio on $S:=\supp(P)\cup\supp(Q)$ as in~\eqref{def.ell}. If \(S\) is an interval with finite left endpoint \(x_*:=\inf S\), and \(\ell\) is unimodal on \(S\), then the following assertions are equivalent:
\begin{enumerate}[label=\textup{(\roman*)}]
\item $P\st Q$.
\item $P\hr Q$.
\item $\ell(x_*+)\ge 1$, where $\ell(x_*+)$ denotes the right limit of \(\ell\) at $x_*$.
\end{enumerate}
\end{thm}

\begin{proof}
We prove \(\textup{(iii)}\Rightarrow\textup{(i)}\), \(\textup{(i)}\Rightarrow\textup{(iii)}\), \(\textup{(i)}\Rightarrow\textup{(ii)}\), and finally \(\textup{(ii)}\Rightarrow\textup{(i)}\). %

Set
\[
d(x):=f_P(x)-f_Q(x),\qquad x\in S.
\]
On $S$, the sign of $d(x)$ agrees with the sign of $\ell(x)-1$ under the conventions in~\eqref{def.ell}. Since $\ell$ is unimodal, the superlevel set
\[
A:=\{x\in S:\ell(x)\ge 1\}
\]
is an interval in $S$. In this case, \(d\) has at most one sign change on \(S\), from nonnegative to nonpositive. 

First, \textup{(iii)} $\Longrightarrow$ \textup{(i)}. Recall that the unimodality of \(\ell\) ensures that the right limit \(\ell(x_*+)\) exists in \([0,\infty]\). Suppose \(\ell(x_*+)\ge 1\). By unimodality, either \(\ell\ge 1\) on some right-neighbourhood of \(x_*\), in which case \(A\) is an initial interval of \(S\); or \(\ell<1\) throughout some such neighbourhood, which forces \(\ell(x_*+)=1\) and \(\ell\le 1\) on all of \(S\), hence \(f_P\le f_Q\) and \(\bar F_P\le \bar F_Q\). In both cases, \(P\st Q\).

It remains to treat the case where \(A\) is an initial interval of \(S\). 
Define
\[
D(x):=P([x_*,x)\cap E)-Q([x_*,x)\cap E) =\int_{[x_*,x)\cap E} d\,\dx[]\mu.
\]
The function $D$ starts at $0$, increases while $d\ge 0$, decreases after the sign change, and returns to $0$ because $\int_E d\,\dx[]\mu=0.$ Hence $D(x)\ge 0$ for every $x\in E$. But $\bar F_P(x)-\bar F_Q(x)=-D(x)$ on $S$, and the comparison is trivial outside $S$. Consequently, $P\st Q$.

Second, \textup{(i)} $\Longrightarrow$ \textup{(iii)}. We proceed by contraposition. Suppose \textup{(iii)} does not hold, that is, \(\ell(x_*+)<1\). Then $A$ does not contain a right-neighbourhood of $x_*$, so there exists $\varepsilon>0$ such that
\[
\ell(x)<1
\qquad\text{for every }x\in S\cap(x_*,x_*+\varepsilon).
\]
Hence $d<0$ on that initial segment, so $D(x)<0$ there. Since $\bar F_P(x)-\bar F_Q(x)=-D(x)$ on $S$, this gives $\bar F_P(x)>\bar F_Q(x)$ on the same segment, implying $P\not\st Q$.

Third, \textup{(i)} $\Longrightarrow$ \textup{(ii)}.
Assume $P\st Q$, and set
\[
J':=\{x\in S:\bar F_Q(x)>0\}.
\]
Then, for every $x\in J'$, $$T(x):=\dfrac{\bar F_P(x)}{\bar F_Q(x)}\le 1.$$ We claim that $T(x)\le \ell(x)$ on $J'$. This is immediate when \(\ell(x)\ge 1\). If \(\ell(x)<1\), then \(x\notin A\). Since \(A\) is an interval and \textup{(i)} implies \textup{(iii)}, the set \(A\) must be an initial interval of \(S\). Therefore
\[
\ell(y)\le \ell(x)\qquad\text{for every }y\in S\cap[x,\infty).
\]
Consequently,
\[
\bar F_P(x) =\int_{[x,\infty)\cap E} f_P\,\dx[]\mu \le \ell(x)\int_{[x,\infty)\cap E} f_Q\,\dx[]\mu =\ell(x)\bar F_Q(x),
\]
and hence $T(x)\le \ell(x)$.

\medskip
\noindent
If $S\subseteq\Z$, then for every $k\in J'$ with $\bar F_Q(k+1)>0$, write
\[
T(k)-T(k+1) = \frac{P(\{k\})+\bar F_P(k+1)}{\bar F_Q(k)} -\frac{\bar F_P(k+1)}{\bar F_Q(k+1)}.
\]
If $Q(\{k\})=0$, then $\bar F_Q(k+1)=\bar F_Q(k)$, and the right-hand side reduces to $P(\{k\})/\bar F_Q(k)\ge 0$. If $Q(\{k\})>0$, the right-hand side equals
\[
\frac{Q(\{k\})}{\bar F_Q(k)}\bigl(\ell(k)-T(k+1)\bigr).
\]
We show that this is nonnegative. When $\ell(k)\ge 1$, then $T(k+1)\le 1\le \ell(k)$. When $\ell(k)<1$, then $k\notin A$, and since $A$ is an initial interval of $S$ and $\ell$ is unimodal, $\ell$ is nonincreasing on $S\setminus A$; therefore $\ell(k+1)\le \ell(k)$. Combining with the already-proved $T(k+1)\le \ell(k+1)$ yields $T(k+1)\le \ell(k)$. In either case $T(k)\ge T(k+1)$, so $T$ is nonincreasing and $P\hr Q$.

\medskip
\noindent
If $S\subseteq\R$, then $\bar F_P$ and $\bar F_Q$ are absolutely continuous, and $\bar F_Q$ is bounded away from $0$ on every compact subinterval of $J'$, so $T=\bar F_P/\bar F_Q$ is absolutely continuous on compact subintervals of $J'$. For a.e.\ $x\in J'$,
\[
T'(x) = \frac{f_Q(x)\bar F_P(x)-f_P(x)\bar F_Q(x)}{\bar F_Q(x)^2}.
\]
Observe that the sign of $T'(x)$ matches that of $T(x)-\ell(x)\le 0$ when $f_Q(x)>0$, and the sign of $-f_P(x)/\bar F_Q(x)\le 0$ when $f_Q(x)=0$. In either case, $T'(x)\le 0$. Consequently, $T$ is nonincreasing on $J'$, so once more $P\hr Q$. This proves the implication.

Finally, \textup{(ii)} $\Longrightarrow$ \textup{(i)} follows from the standard implication $P\hr Q\Longrightarrow P\st Q$ \cite{ShakedShanthikumar}.
\end{proof}

The strongest natural hypothesis on $\ell$ is relative log-concavity. Under it, Theorem~\ref{thm:unimodal-criterion} specialises to a single endpoint check for $\hr$ and $\st$, and the $\lr$ direction reduces to a single endpoint check as well, in distinct discrete and continuous forms.

\begin{cor}\label{cor:lc-endpoints}
Let \(P\) and \(Q\) be probability measures on \(E\) with common support interval \(J\subseteq E\) and finite left endpoint \(x_0:=\inf J\). If \(P\lc Q\), then \(\ell\) is unimodal on \(J\) and
\[
P\st Q \iff P\hr Q \iff \ell(x_0+)\ge 1.
\]
\end{cor}

\begin{proof}
Since \(P\lc Q\), \(\ell\) is log-concave on \(J\), hence unimodal \cite{Karlin1968}. Applying Theorem~\ref{thm:unimodal-criterion} with $S=J$ and $x_*=x_0$ gives the equivalence.
\end{proof}

Building on this result, the likelihood-ratio order reduces to a single-point endpoint test.
\begin{cor}\label{cor:lc-lr}
Assume the hypotheses of Corollary~\ref{cor:lc-endpoints} hold. If  \(x_0\in\supp(P)\), then
\[
P\lr Q \iff
\begin{dcases}
\ell'(x_0+)\le 0 & E\subseteq\R \\
\ell(x_0+1) - \ell(x_0)\le 0 & E\subseteq\N_0,
\end{dcases}
\]
\end{cor}

\begin{proof}
Since \(P\lc Q\), \(\ell\) is unimodal as established in the previous proof. The assumption \(x_0\in\supp(P)\) gives \(\ell(x_0)>0\).

First, consider the case \(E\subseteq\N_0\). Concavity of \(\log\ell\) implies \(\Delta\log\ell(k)\) nonincreasing in \(k\). Therefore, \(\Delta\log\ell(k)\le 0\) for every admissible \(k\in J\) if and only if $\Delta\log\ell(x_0)\le 0.$ It follows that \(\ell(x_0+1)\le \ell(x_0)\) since \(\ell(x_0)>0\). Hence, \(P\lr Q\). The case \(E\subseteq\R\) is analogous. Concavity of \(\log\ell\) makes its right derivative nonincreasing on \(J\). Therefore, \((\log\ell)'(x)\le 0\) on \(J\) if and only if \((\log\ell)'(x_0+)\le 0\). Since \(\ell(x_0)>0\), the sign of  \((\log\ell)'(x_0+)\) coincides with that of \(\ell'(x_0+)\). Hence, \((\log\ell)'(x_0+)\le 0\) coincides with \(\ell'(x_0+)\le 0\), and therefore to \(P\lr Q\).
\end{proof}

\begin{example}[Gamma versus Gamma]
\label{ex:gamma-vs-gamma}
For \(P=\Gam(r_1,\beta_1)\) and \(Q=\Gam(r_2,\beta_2)\) on \((0,\infty)\), where \(\Gam(r,\beta)\) denotes the gamma law with shape \(r\) and scale \(\beta\), we have
\[
\ell(x)=\frac{\Gamma(r_2)\beta_2^{r_2}}{\Gamma(r_1)\beta_1^{r_1}} x^{r_1-r_2}\exp\!\left(-(\beta_1^{-1}-\beta_2^{-1})x\right),\qquad x>0.
\]
and
\[
\partial_x^2\log\ell(x)=-\frac{r_1-r_2}{x^2},
\]
Hence, \(P\lc Q\) if and only if \(r_1\ge r_2\). Under that assumption, Corollary~\ref{cor:lc-endpoints} gives
\[
\Gam(r_1,\beta_1)\st\Gam(r_2,\beta_2) \iff \Gam(r_1,\beta_1)\hr\Gam(r_2,\beta_2) \iff \ell(0+)\ge 1.
\]
Now \(\ell(0+)=0\) when \(r_1>r_2\), while \(r_1=r_2=r\) implies that  $\ell(0+)=\left(\frac{\beta_2}{\beta_1}\right)^r.$ Therefore, under the relative log-concavity condition \(r_1\ge r_2\),
\begin{align*}
\Gam(r_1,\beta_1)\st\Gam(r_2,\beta_2)
&\iff \Gam(r_1,\beta_1)\hr\Gam(r_2,\beta_2)\\
&\iff r_1=r_2\ \text{and}\ \beta_1\le \beta_2.
\end{align*}
\end{example}

The next example illustrates a comparison that is unimodal but not relative-log-concave, so Theorem~\ref{thm:unimodal-criterion} applies but Corollary~\ref{cor:lc-endpoints} does not.

\begin{example}[Half-Student in degrees of freedom]
\label{ex:half-student}
Let $|X_\nu|$ be the half-Student law on $[0,\infty)$ with $\nu$ degrees of freedom. The density is
\[
f_\nu(x)=\frac{2\,\Gamma((\nu+1)/2)}{\sqrt{\nu\pi}\,\Gamma(\nu/2)}\Bigl(1+\frac{x^2}{\nu}\Bigr)^{-(\nu+1)/2}.
\]
For $\nu_1<\nu_2$, a direct calculation gives
\[
\partial_x\log\ell(x) =\frac{(\nu_2-\nu_1)\,x(1-x^2)}{(\nu_1+x^2)(\nu_2+x^2)},
\]
which is positive on $(0,1)$ and negative on $(1,\infty)$. Thus, $\log\ell$ increases up to $x=1$ and decreases afterwards, so both $\log\ell$ and $\ell$ are unimodal on $[0,\infty)$, with a unique mode at $x=1$. On the other hand, $\log\ell$ is not concave on $[0,\infty)$: its right derivative at $0$ is $0$, but it becomes positive immediately to the right of $0$, which cannot happen for a concave function. Therefore Theorem~\ref{thm:unimodal-criterion} applies, whereas Corollary~\ref{cor:lc-endpoints} does not. The endpoint condition then becomes
\[
|X_{\nu_2}|\st|X_{\nu_1}|\iff |X_{\nu_2}|\hr|X_{\nu_1}|\iff \ell(0)\ge 1\iff f_{\nu_2}(0)\ge f_{\nu_1}(0),
\]
and $\nu\mapsto f_\nu(0)$ is increasing on $(0,\infty)$, since concavity of $\psi$ together with the recurrence $\psi(z+1)-\psi(z)=1/z$ gives $\psi(z+\tfrac{1}{2})-\psi(z)>1/(2z)$. Hence, larger degrees of freedom give a smaller half-Student random variable in both the usual stochastic and hazard-rate orders:
\[
\nu_1<\nu_2 \quad\Longrightarrow\quad |X_{\nu_2}|\st|X_{\nu_1}|
\quad\text{and}\quad
|X_{\nu_2}|\hr|X_{\nu_1}|.
\]
\end{example}

\subsection{Sign-pattern endpoint criteria}

The unimodality hypothesis can be relaxed to a sign-pattern condition on \(\ell-1\). \cite{YaoLouPan2023} proved the corresponding endpoint reduction for nonnegative random variables. In our notation, their criterion assumes that, on \([0,\infty)\), \(\log \ell\) has at most two roots and is negative beyond the largest root. Under that hypothesis, the hazard-rate and usual stochastic orders reduce to the single endpoint check \(\ell(0)\ge 1\), with the hazard-rate part also using the same right-tail monotonicity as in Proposition~\ref{prop:ylp-kernel}.

Recall that on $(0,\infty)$ the zero sets of \(\log\ell\) and of \(\ell-1\) coincide, and so do their sign patterns. The Yao--Lou--Pan hypothesis therefore amounts to two requirements: a sign-pattern condition (at most two sign changes of \(\ell-1\), negative on the rightmost piece) and a regularity condition (the zero set is discrete, i.e., at most two points). The support-interval version below relaxes the regularity requirement, retaining only the sign pattern, which is the only ingredient the conditional tail-expectation argument actually uses.

\begin{lem}
\label{lem:tail-mean-sign}
Let $\mu$ be a probability law on an ordered support $J \subseteq E$ with finite least element $x_0$, and let $\phi:J\to\R$ be $\mu$-integrable with $\E[\phi(X)]=0$. If $\phi$ has at most two sign changes on $J$ and is nonpositive on the rightmost part, then
\[
\E[\phi(X)\mid X\ge x]\le 0 \quad\text{for every }x\in J\text{ with }\bar F_\mu(x)>0 \quad\Longleftrightarrow\quad \phi(x_0)\ge 0,
\]
where the forward implication additionally requires that, in the case $\phi(x_0)=0$, $\phi$ is nonnegative on a right-neighbourhood of $x_0$.
\end{lem}
\begin{proof}
Set $D(x):=\E[\phi(X)\mid X\ge x]$ on $\{x\in J:\bar F_\mu(x)>0\}$, and $I(x):=\int_{[x_0,x)}\phi\,\dx[]\mu$. Multiplying $D(x)$ by $\bar F_\mu(x)$ gives
\[
D(x)\bar F_\mu(x) =\E_\mu\bigl[\phi(X)\one\{X\ge x\}\bigr] =-I(x),
\]
where the last equality uses $\E[\phi(X)]=0$. Hence, $D(x)$ and $I(x)$ have opposite signs.

The function $I$ vanishes at both ends of $J$: $I(x_0)=0$ trivially, and $\lim_{x\uparrow\sup J}I(x)=\int_J\phi\,\dx[]\mu=0$. Moreover, $I$ is continuous (or piecewise constant in the discrete case) and monotone on each maximal interval where $\phi$ has constant sign (nondecreasing where $\phi\ge 0$ and nonincreasing where $\phi\le 0$). 

Suppose first that \(\phi(x_0)\ge 0\). Under the right-neighbourhood hypothesis when \(\phi(x_0)=0\), the sign-change bound gives a point \(b\in J\) such that \(\phi\ge 0\) on \([x_0,b]\cap J\) and \(\phi\le 0\) on \([b,\sup J)\cap J\). Zero intervals between the two pieces are absorbed into either side. Hence, $I$ starts at $0$, increases while $\phi\ge 0$, then decreases while $\phi\le 0$, and ends once more at $0$. Consequently, $I(x)\ge 0$ for every $x\in J$, so $D(x)\le 0$ wherever it is defined.

Conversely, if $\phi(x_0)<0$, then $I(x)<0$ on some right-neighbourhood of $x_0$ where $\bar F_\mu(x)>0$, and therefore $D(x)>0$ there.
\end{proof}
\begin{rem}
The forward direction of this lemma is the case \(n=0\) of Theorem~5.4 in Chapter~XI of~\cite{KarlinStudden1966}, with \(u_0\equiv 1\) and signed measure \(-\phi\,\dx[]\mu\), applied to tail indicators \(\one\{t\ge x\}\).
\end{rem}

\begin{prop}
\label{prop:ylp-kernel}
Let \(P,Q\) be probability laws on \(E\) whose common support interval \(J\subseteq E\) has finite left endpoint \(x_0:=\min J\), and let \(\ell\) be the likelihood ratio on \(J\), assumed continuous when \(J\subseteq\R\). Suppose \(\phi:=\ell-1\) has at most two sign changes on \(J\) and is negative
on the final sign-constant part of the interval. If \(\ell(x_0)=1\), assume in addition that \(\ell\ge 1\) on a right-neighbourhood of \(x_0\), then
\[
P\st Q \iff \ell(x_0)\ge 1.
\]
If, in addition, $\ell$ is nonincreasing on the rightmost sign-constant part of the interval of $\phi$ (equivalently, $\ell'(x)\le 0$ there when $J\subseteq\R$, or $\Delta\ell(k)\le 0$ there when $J$ is an integer interval), the same equivalent conditions characterise $P\hr Q$.
\end{prop}

\begin{proof}
Working under $Q$, the function $\phi(x):=\ell(x)-1$ satisfies
\[
\E_Q[\phi(X)]=\int_J\bigl(f_P-f_Q\bigr)\,\dx[]\mu=0, \qquad \E_Q[\phi(X)\mid X\ge x]=\frac{\bar F_P(x)}{\bar F_Q(x)}-1
\]
on $\{x\in J:\bar F_Q(x)>0\}$, and inherits from $\ell$ at most two sign changes on $J$ with $\phi<0$ on the rightmost sign-constant interval. Therefore,  $P\st Q\iff \E_Q[\phi(X)\mid X\ge x]\le 0$, holds for every $x\in J\text{ with }\bar F_Q(x)>0$. Applying Lemma~\ref{lem:tail-mean-sign} to $(\phi,Q)$ gives
\[
P\st Q\iff \phi(x_0)\ge 0\iff \ell(x_0)\ge 1.
\]
This is the equivalence (i)$\Leftrightarrow$(ii); the converse direction of Lemma~\ref{lem:tail-mean-sign} furnishes the implication $\ell(x_0)<1\Rightarrow P\not\st Q$ directly.

For the $\hr$ part, on $\{x\in J:\bar F_Q(x)>0\}$,
\[
P\hr Q\iff \frac{f_P(x)}{\bar F_P(x)}\ge\frac{f_Q(x)}{\bar F_Q(x)} \iff \ell(x)\ge\frac{\bar F_P(x)}{\bar F_Q(x)},
\]
that is, $\phi(x)\ge \E_Q[\phi(X)\mid X\ge x]$. Assume $\ell(x_0)\ge 1$ and $\ell$ is nonincreasing on the rightmost sign-constant interval of $\phi$. If $\phi(x)\ge 0$, the $\st$ part of the proposition gives $\E_Q[\phi(X)\mid X\ge x]\le 0\le \phi(x)$. If $\phi(x)<0$, then $x$ lies on the rightmost  (sign-constant) part of the interval, so $\phi(u)\le \phi(x)$ for every $u\ge x$, hence $\E_Q[\phi(X)\mid X\ge x]\le \phi(x)$. Consequently, $P\hr Q$. The converse follows from the standard implication $P\hr Q\Longrightarrow P\st Q$ \cite{ShakedShanthikumar}.
\end{proof}

A direct alternative to Proposition~\ref{prop:ylp-kernel} is to control only the shape of the superlevel set \(\{\ell\ge 1\}\) at the left endpoint. The next result specialises Proposition~\ref{prop:ylp-kernel} to a single sign change of $f_P-f_Q$, in exchange for dropping the continuity hypothesis on $\ell$ when $J\subseteq\R$.

\begin{prop}
\label{prop:ylp-one-sided-score}
Let $P,Q$ be probability laws on $E$ whose common support interval $J\subseteq E$ has finite left endpoint $x_0:=\min J$. If $\ell(x_0)\ge 1$ and the superlevel set
\[
A:=\{x\in J:\ell(x)\ge 1\}
\]
is an interval, then $P\st Q$. If, in addition, $\ell$ is nonincreasing on $J\setminus A$, then $P\hr Q$.
\end{prop}

\begin{proof}
Since $\ell(x_0)\ge 1$, the left endpoint belongs to $A$, so the interval assumption gives $A=[x_0,r]\cap J$ for some $r\in J\cup\{\sup J\}$. Hence $\ell\ge 1$ on $A$ and $\ell\le 1$ on $J\setminus A$, which means $f_P-f_Q$ has at most one sign change on $J$, from $+$ to $-$. The argument from~\textup{(iii)}\,$\Rightarrow$\,\textup{(i)} of Theorem~\ref{thm:unimodal-criterion} (using only the single sign change of $f_P-f_Q$, not the unimodality of $\ell$) gives $P\st Q$.

For the $\hr$ direction, set $T(x):=\bar F_P(x)/\bar F_Q(x)$ on $\{x\in J:\bar F_Q(x)>0\}$. On $A$ we have $\ell\ge 1$, so $f_P\ge f_Q$, and combining with $\bar F_P\le \bar F_Q$, established in the first part, gives $h_P\ge h_Q$. On $J\setminus A$, $\ell$ is nonincreasing and $\ell\le 1$. The bound $T(x)\le \ell(x)$ on $J\setminus A$ follows from the same tail-comparison proof of \textup{(i)}\,$\Rightarrow$\,\textup{(ii)} in Theorem~\ref{thm:unimodal-criterion}\textup{(i)}\,$\Rightarrow$\,\textup{(ii)}: for $x\in J\setminus A$ and $y\ge x$, monotonicity of $\ell$ on $J\setminus A$ gives $\ell(y)\le\ell(x)$, so $\bar F_P(x)\le \ell(x)\bar F_Q(x)$. That is, $T(x)\le \ell(x)$. The discrete and continuous monotonicity arguments of that proof then yield $T$ nonincreasing on $J\setminus A$. Hence, $h_P\ge h_Q$ holds also on $J\setminus A$.
\end{proof}

\begin{example}
\label{ex:ylp-comparison}
\begin{figure}[ht]
\centering
\includegraphics[width=\linewidth]{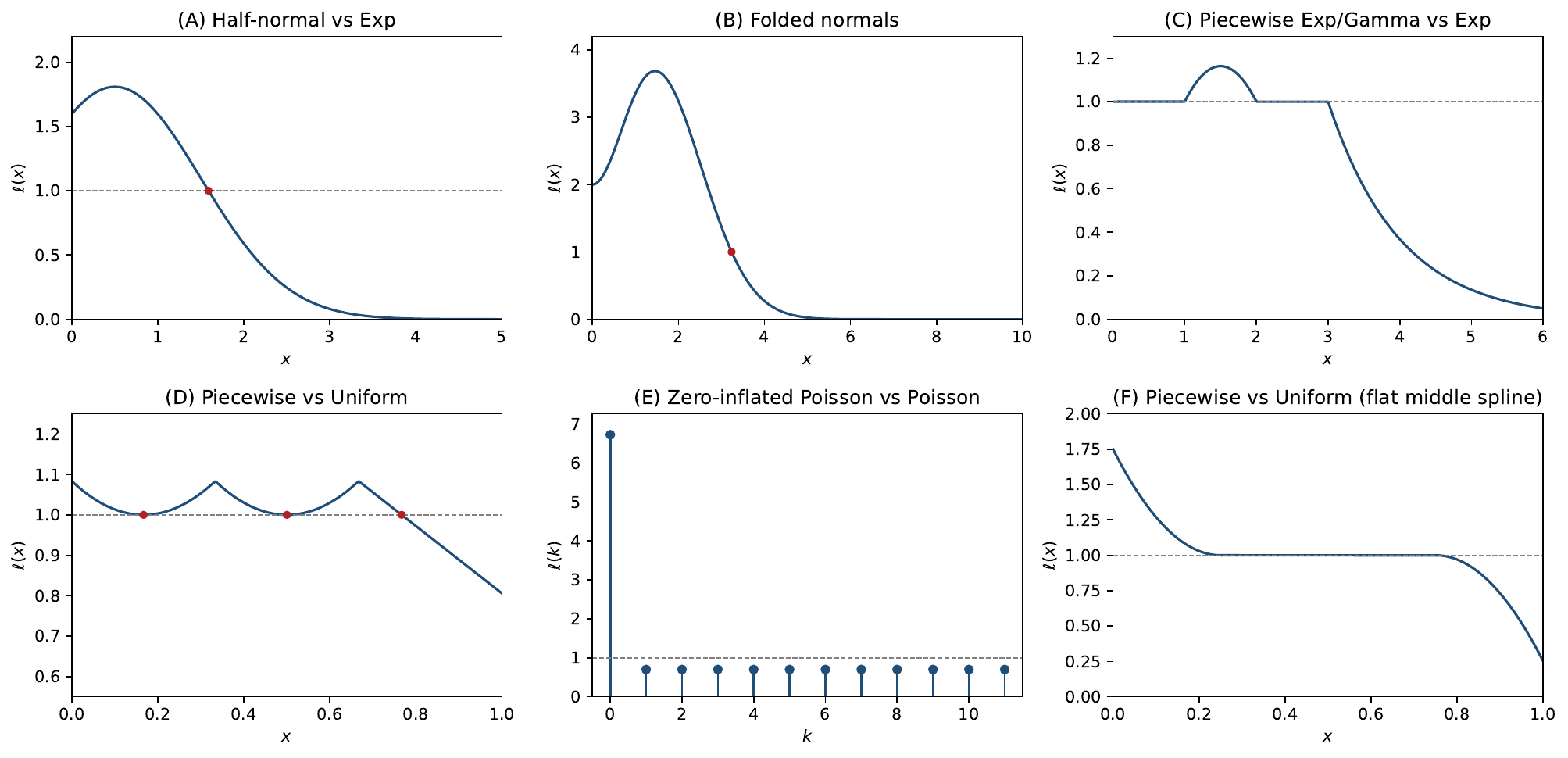}
\caption{Likelihood-ratio shapes from Example~\ref{ex:ylp-comparison}.}
\label{fig:ylp-shapes}
\end{figure}
Theorems~2.1 and~2.3 of \cite{YaoLouPan2023} assume, for nonnegative random variables, that \(\log \ell\) has at most two roots and is negative beyond the largest root; by Remark~2.5 in~\cite{YaoLouPan2023}, the same endpoint reduction holds on a finite support interval \([0,a)\). For brevity, we refer to these original hypotheses as the
Yao--Lou--Pan (YLP) condition. Figure~\ref{fig:ylp-shapes} exhibits six representative situations. 
\begin{enumerate}[label=\textup{(\Alph*)}]
\item \emph{\(\lc\) and YLP hold, but \(\lr\) does not.}
For \(P=\HN(1)\) (half-normal with scale~\(1\)) and \(Q=\Exp(\tfrac12)\) (exponential with rate~\(\tfrac12\)). Easy calculation gives  $\log\ell(x)=\log\!\sqrt{\frac{8}{\pi}}-\frac{x^2}{2}+\frac{x}{2},$ which is concave with a unique interior maximum. Thus \(P\lc Q\) and the original YLP hypothesis holds, but \(P\not\lr Q\).

\item \emph{YLP-type case.}
The folded-normal pair \(|Z_i|\) with \(Z_i\sim N(\mu_i,\sigma_i^2)\), \(i=1,2\), is the illustrative  example of \cite{YaoLouPan2023} (Remark~2.2, Example~I): they invoke Theorem~1 of \cite{WangWang2011} for the \(\st\) part and Proposition~3 of the same reference for the behaviour of the sign of \(\ell'\) on \((0,\infty)\). Together, these place the comparison outside relative log-concavity but inside their two-roots framework. Proposition~\ref{prop:ylp-kernel} reaches the same conclusion via the sign pattern of \(\ell-1\), and so applies directly here.

\item \emph{YLP fails because \(\log\ell\) vanishes on an interval. Also \(\lr\) does not hold.}
Take \(Q=\Exp(1)\), \(0<a_1<a_2<a_3\), \(\mu>1\), and
\[
f_P(x)=
\begin{cases}
\exp(-x), & 0\le x\le a_1,\\[1mm]
\bigl[1+c\,(x-a_1)(a_2-x)\bigr]\exp(-x), & a_1\le x\le a_2,\\[1mm]
\exp(-x), & a_2\le x\le a_3,\\[1mm]
\exp((\mu-1)a_3)\exp(-\mu x), & x\ge a_3,
\end{cases}
\]
with \(c>0\) determined by \(\int f_P=1\). Then \(\ell\equiv 1\) on \([0,a_1]\cup[a_2,a_3]\), \(\ell>1\) on \((a_1,a_2)\), and \(\ell<1\) on \((a_3,\infty)\). The likelihood-ratio order does \emph{not} hold here, since \(\ell\) is not monotone and YLP fails because \(\log\ell\) vanishes on an interval. However, Proposition~\ref{prop:ylp-kernel} still applies. 
\item \emph{YLP fails because \(\log\ell\) has more than two zeros.}
Take \(Q=U[0,1]\) and let \(P\) have density
\[
f_P(x)=
\begin{cases}
1+\varepsilon(x-\tfrac16)^2, & 0\le x\le \tfrac13,\\[1mm]
1+\varepsilon(x-\tfrac12)^2, & \tfrac13\le x\le \tfrac23,\\[1mm]
1+\tfrac{\varepsilon}{36}-\tfrac{5\varepsilon}{18}(x-\tfrac23), & \tfrac23\le x\le 1,
\end{cases}
\qquad \varepsilon=3.
\]
Then \(\ell=f_P\) touches the level \(1\) at \(x=\tfrac16\) and \(x=\tfrac12\), and crosses it at \(x=\tfrac{23}{30}\). Thus \(l=\log\ell\) has three zeros, so the YLP root-count fails, while Proposition~\ref{prop:ylp-kernel} still applies because \(\ell-1\) has sign pattern \(+,-\). Again \(P\not\lr Q\).

\item \emph{Discrete case, YLP does not apply.}
Fix \(0<\pi<1\) and \(\lambda>0\). Let \(P\) be the zero-inflated Poisson law with $P(\{0\})=\pi+(1-\pi)e^{-\lambda}$, and  $P(\{k\}) \propto (1-\pi)e^{-\lambda}{\lambda^k}$ for $k\ge 1$.  Let \(Q=\Poi(\lambda)\) denote the Poisson law with mean \(\lambda\). Then
\[
\ell(0)=1-\pi+\pi e^\lambda>1,\qquad
\ell(k)=1-\pi<1,\quad k\ge 1.
\]
Hence \(\{\ell\ge 1\}=\{0\}\), so Proposition~\ref{prop:ylp-one-sided-score} applies directly. 

\item \emph{YLP root-count fails but \(\lr\) holds.}
Take \(Q=U[0,1]\) and let \(P\) have density
\[
f_P(x)=
\begin{cases}
1+\varepsilon(1-4x)^2, & 0\le x\le \tfrac14,\\[1mm]
1, & \tfrac14\le x\le \tfrac34,\\[1mm]
1-\varepsilon(4x-3)^2, & \tfrac34\le x\le 1,
\end{cases}
\qquad 0<\varepsilon<1.
\]
Observe that \(\ell=f_P\) is nonincreasing. Thus, \(P\lr Q\), but \(l=\log\ell\) vanishes on the whole interval \([\tfrac14,\tfrac34]\). The hypothesis of  at most two roots therefore fails even though $\ell$ is unimodal, whereas Proposition~\ref{prop:ylp-kernel} still applies because the sign pattern of \(\ell-1\) is \(+,-\).
\end{enumerate}
\end{example}

\section{Conclusion}
\label{sec:discussion}

We have shown that the left-endpoint reduction for the hazard-rate and usual stochastic orders persists well beyond relative log-concavity of the likelihood ratio. Besides the unimodal case, it remains valid under a sign-pattern condition and under a direct superlevel-set criterion that also covers discontinuous likelihood ratios.

\printbibliography

\end{document}